\documentclass[11pt]{article}

\usepackage{amsmath,amssymb,amsfonts,diagrams}
\newenvironment{keywords}{\noindent\small {\it Keywords\/}:}{\vskip 4pt}
\newenvironment{classification}{\noindent\small 2000 {\it Mathematics Subject
Classification\/}:}{\vskip 12pt}

\newcommand{\comps}{{\mathbb C}}

\newcommand{\ints}{{\mathbb Z}}
\newcommand{\posints}{{\mathbb N}}

\newcommand{\torus}{{\mathbb T}}

\newcommand{\tensor}{\otimes}

\newcommand{\Tensor}{\hat{\otimes}}

\newcommand{\wTensor}{\check{\otimes}}

\newcommand{\cstar}{{C^\ast}}

\newcommand{\id}{{\mathrm{id}}}

\newcommand{\lspan}{{\operatorname{span}}}

\newcommand{\A}{{\mathfrak A}}
\newcommand{\B}{{\mathfrak B}}
\newcommand{\Hilbert}{{\mathfrak H}}

\newcommand{\M}{{\mathfrak M}}

\newcommand{\SL}{\operatorname{SL}}

\newcommand{\cl}[1]{#1^-}
\newcommand{\varcl}[1]{\overline{#1}}

\renewcommand{\baselinestretch}{1.2}
\newcommand{\dated}{\mbox{} \hfill {\small [{\tt \today}]}} 
\usepackage{amsthm,enumerate}
\theoremstyle{plain}
\newtheorem{theorem}{Theorem}[section]
\newtheorem{lemma}[theorem]{Lemma}
\newtheorem{corollary}[theorem]{Corollary}

\theoremstyle{definition}
\newtheorem{definition}[theorem]{Definition}
\theoremstyle{remark}
\newtheorem*{remark}{Remark}
\newtheorem*{example}{Example}
\newtheorem*{rems}{Remarks}
\newtheorem*{exs}{Examples}

\theoremstyle{plain}
\newtheorem*{taklem}{Ozawa's Lemma}

\theoremstyle{definition}
\newtheorem{question}{Question}

\begin{document}

\title{(Non-)amenability of ${\cal B}(E)$}
\author{\textit{Volker Runde}}
\date{}

\maketitle

\begin{abstract}
In 1972, the late B.\ E.\ Johnson introduced the notion of an amenable Banach algebra and asked whether the Banach algebra ${\cal B}(E)$ of all bounded linear operators on a Banach space $E$ could ever be amenable if $\dim E = \infty$. Somewhat surprisingly, this question was answered positively only very recently as a by-product of the Argyros--Haydon result that solves the ``scalar plus compact problem'': there is an infinite-dimensional Banach space $E$, the dual of which is $\ell^1$, such that ${\cal B}(E) = {\cal K}(E)+ \comps \, \id_E$. Still, ${\cal B}(\ell^2)$ is not amenable, and in the past decade, ${\cal B}(\ell^p)$ was found to be non-amenable for $p=1,2,\infty$ thanks to the work of C.\ J.\ Read, G.\ Pisier, and N.\ Ozawa. We survey those results, and then---based on joint work with M.\ Daws---outline a proof that establishes the non-amenability of ${\cal B}(\ell^p)$ for all $p \in [1,\infty]$.
\end{abstract}

\section{Amenable Banach algebras} 

The phenomenon of amenability manifests itself in many areas of mathematics. Originally, amenability was a property of groups and semigroups (\cite{Pat}), but nowadays there are amenable representations, Banach algebras, von Neumann algebras, etc.
\par 
The notion of an amenable Banach algebra originates in B.\ E.\ Johnson's memoir \cite{Joh1}. The link between amenable locally compact groups and amenable Banach algebras is \cite[Theorem 2.5]{Joh1}: a locally compact group $G$ is amenable if and only if its group algebra $L^1(G)$ is an amenable Banach algebra. 
\par 
We briefly recall, not the original definition from \cite{Joh1}, but an equivalent characterization from \cite{Joh2}. Given a Banach algebra $\A$, the projective tensor product $\A \Tensor \A$ is a Banach $\A$-bimodule via
\[
  a \cdot (x \tensor y) := ax \tensor y \quad\text{and}\quad (x \tensor y) \cdot a :=x \tensor ya \qquad (a,x,y \in \A),
\]
so that the linear map $\Delta \!: \A \Tensor \A \to \A$ induced by multiplication becomes an $\A$-bimodule homomorphism.
\begin{definition}
A Banach algebra $\A$ is called \emph{amenable} if there is an \emph{approximate diagonal} for $\A$, i.e., a bounded net $( \boldsymbol{d}_\alpha )_\alpha$ in $\A \Tensor \A$ such that
\[
  a \cdot \boldsymbol{d}_\alpha - \boldsymbol{d}_\alpha \cdot a \to 0 \qquad (a \in \A)
\]
and 
\[
  a \Delta \boldsymbol{d}_\alpha \to a \qquad (a \in \A).
\]
\end{definition}
\par
For a modern account of the theory of amenable Banach algebras, see \cite{LoA}, for instance.
\begin{example}
Let $E$ be a \emph{finite-dimensional} Banach space. Then ${\cal B}(E)$, the algebra of all bounded linear operators on $E$ is isomorphic to the full matrix algebra $M_n$, where $n = \dim E$. Let $\{ e_{j,k} : j,k=1, \ldots, n \}$ be a set of matrix units for $M_n$. Then
\[
  \boldsymbol{d} := \frac{1}{n} \sum_{j,k=1}^n e_{j,k} \tensor e_{k,j} 
\]
satisfies
\[
  a \cdot \boldsymbol{d} = \boldsymbol{d} \cdot a \quad (a \in M_n) \qquad\text{and}\qquad \Delta \boldsymbol{d} = I_n,
\]
where $I_n$ is the identity matrix in $M_n$. We call $\boldsymbol{d}$ a \emph{diagonal} for $M_n \cong {\cal B}(E)$. In particular, ${\cal B}(E)$ is amenable.
\end{example}
\par
The memoir \cite{Joh1} concludes with 13 suggestions for further research, in particular (\cite[10.4]{Joh1})
\begin{quote}
\textit{Is ${\cal K}(E)$---\emph{the Banach algebra of all compact operators}---amenable for all Banach spaces $E$? Is ${\cal B}(E)$ ever amenable for any infinite-dimensional $E$?}
\end{quote}
and (\cite[10.2]{Joh2})
\begin{quote}
[\dots] \textit{Do there exist non-amenable $\cstar$-algebras, in particular is ${\cal B}(\Hilbert)$---\emph{where $\Hilbert$ is a Hilbert space}---amenable?}
\end{quote}
\par 
From a philosophical point of view, the answer to the question whether ${\cal B}(E)$ can be amenable for any infinite-dimensional $E$ ought to be a clear ``no'': a common feature of all the various manifestations of amenability is that amenable objects tend to be ``small''---whatever that may mean precisely---and ${\cal B}(E)$ simply feels too ``large'' to be amenable.

\section{Amenability of ${\cal K}(E)$ and the ``scalar plus compact problem''}

If $E$ is a Banach space with the approximation property, then ${\cal K}(E)$ is the closure of the finite rank operators in the operator norm and thus appears to be ``small'' enough to be amenable, at least for certain $E$. Indeed, already Johnson showed in \cite{Joh1} that ${\cal K}(\ell^p)$ is amenable for $p \in (1,\infty)$ (\cite[Proposition 6.1]{Joh1}) and that ${\cal K}({\cal C}(\torus))$, with $\torus$ denoting the unit circle, is amenable (\cite[Proposition 6.3]{Joh1}).
\par 
The amenability of ${\cal K}(E)$ was further investigated in \cite{GJW}. 
\par 
Recall that a finite \emph{bi-orthogonal} system over a Banach space $E$ is a set $\{ (x_j,x^\ast_k) : j,k=1, \ldots, n \} \subset E \times E^\ast$ such that $\langle x_j, x_k^\ast \rangle = \delta_{j,k}$ for $j,k=1, \ldots, n$. For a matrix $a = [ \alpha_{j,k} ]_{j,k=1}^n \in M_n$ define
\[
  \theta_\alpha (a) := \sum_{j,k=1}^n \alpha_{j,k} x_j \tensor x_k^\ast \in {\cal B}(E).
\]
Then $\theta \!: M_n \to {\cal B}(E)$ is an algebra homomorphism with range in the finite rank operators ${\cal F}(E)$.
\par 
The following is \cite[Definition 4.1]{GJW}:
\begin{definition}
A Banach space $E$ has \emph{property $(\mathbb{A})$} if there is a net $( \{ (x_{j,n_\lambda},x^\ast_{k,n_\lambda}) : j,k=1, \ldots, n_\lambda \})_{\lambda \in \Lambda}$ of finite bi-orthogonal systems over $E$ and a corresponding net of algebra homomorphism $\theta_\lambda \!: M_{n_\lambda} \to {\cal F}(E)$ with the following properties:
\begin{enumerate}[$\mathbb{A}$(i)]
\item $\theta_\lambda(I_{n_\lambda}) \to \id_E$ uniformly on compact subsets of $E$;
\item $\theta_\lambda(I_{n_\lambda})^\ast \to \id_{E^\ast}$ uniformly on compact subsets of $E^\ast$;
\item for each $\lambda \in \Lambda$, there is a finite group $G_\lambda$ of invertible elements of $M_{n_\lambda}$ with $M_{n_\lambda} = \lspan \, G_\lambda$ such that
\[
  \sup_{\lambda \in \Lambda} \sup_{x \in G_\lambda} \| \theta_\lambda(x) \| < \infty.
\]
\end{enumerate}
\end{definition}
\begin{example}
For every compact Hausdorff space $K$, the Banach space ${\cal C}(K)$ has property $(\mathbb{A})$; also, all $L^p$-spaces with $p \in [1,\infty]$ have property $(\mathbb{A})$ (\cite[Theorem 4.7]{GJW}).
\end{example}
\par 
The definition of property $(\mathbb{A})$ may appear overly technical at the first glance, but the simple idea behind it is to build an approximate diagonal for ${\cal K}(E)$ out of diagonals for the full matrix algebras $M_{n_\lambda}$. This means (\cite[Theorem 4.2]{GJW}):
\begin{theorem}
Let $E$ be a Banach space with property $(\mathbb{A})$. Then ${\cal K}(E)$ is amenable.
\end{theorem}
\par 
This establishes, in particular, the amenability for ${\cal K}(E)$ for every $L^p$-space $E$ with $p \in [1,\infty]$. For further results on the amenability (and non-amenability) of ${\cal K}(E)$---or rather: $\varcl{{\cal F}(E)}$---, we refer to \cite{GJW} and \cite{BG}. 
\par 
For many Banach spaces $E$, it is easy to come up with many bounded, non-compact operators: for example, if $E$ has an unconditional basis---so that we can think of operators on $E$ as matrices---, then diagonal matrices and permutation matrices provide examples of operators in ${\cal B}(E) \setminus {\cal K}(E)$ if $\dim E = \infty$. Still, for a general Banach space $E$, it can be difficult to find operators that aren't compact (besides the identity and its scalar multiples).
\par 
The following problem has been known as the ``scalar plus compact problem'':
\begin{quote}
\textit{Is there an infinite-dimensional Banach space $E$ such that ${\cal B}(E) = {\cal K}(E) + \comps \, \id_E$?}
\end{quote}
\par 
Somewhat surprisingly, this problem was recently solved affirmatively by S.\ A.\ Argyros and R.\ G.\ Haydon (\cite{AH}):
\begin{theorem}
There is a Banach space $E$ with $E^\ast = \ell^1$ such that ${\cal B}(E) = {\cal K}(E) + \comps \, \id_E$.
\end{theorem}
\par 
As H.\ G.\ Dales pointed out, this also settles affirmatively the question of whether ${\cal B}(E)$ can be amenable for an infinite-dimensional Banach space:
\begin{corollary} \label{spccor}
There is an infinite-dimensional Banach space $E$ such that ${\cal B}(E)$ is amenable. 
\end{corollary}
\begin{proof}
Let $E$ be an infinite-dimensional Banach space with $E^\ast = \ell^1$ and ${\cal B}(E) = {\cal K}(E) + \comps \, \id_E$. By \cite[Theorem 4,7]{GJW}, $\ell^1$ has property $(\mathbb{A})$ and so has its predual $E$ by \cite[Theorem 4.3]{GJW}. Consequently, ${\cal K}(E)$ is amenable as is its unitization ${\cal B}(E)$.
\end{proof}
\par 
The crucial feature in the proof of Corollary \ref{spccor} is, of course, that ${\cal B}(E)$---contrary to intuition---can be very ``small'' relative to ${\cal K}(E)$.

\section{The Hilbert space case}

The question of wether ${\cal B}(\Hilbert)$ was amenable for an infinite-dimensional Hilbert space $\Hilbert$ was settled relatively soon after it had been posed, as a by-product of some penetrating research on $\cstar$-algebras and their enveloping von Neumann algebras.
\par 
If $\A$ is a $\cstar$-algebra, then so is the algebra $M_n(\A)$ for each $n \in \posints$. We call a linear map $T \!: \A \to \B$ between two $\cstar$-algebras \emph{completely positive} if $\id_{M_n} \tensor T \!: M_n(\A) \to M_n(\B)$ is positive for each $n \in \posints$.
\begin{definition}
A $\cstar$-algebra is called \emph{nuclear} if there are nets $(n_\lambda)_\lambda$ of positive integers and of completely positive contractions $S_\lambda \!: \A \to M_{n_\lambda}$ and $T_\lambda \!: M_{n_\lambda} \to \A$ such that $T_\lambda \circ S_\lambda \to \id_\A$ pointwise on $\A$.
\end{definition}
\par 
This is not the original definition (\cite[Definition XV.1.4]{Tak}) of nuclearity, but it's equivalent to it (\cite[Theorem XV.1.7]{Tak}). Loosely speaking, nuclearity means that the identity factors approximately through full matrix algebras. This suggests that nuclear $\cstar$-algebras shouldn't be too ``large'', which puts nuclearity close to amenability, and indeed:
\begin{theorem} \label{nucam}
The following are equivalent for a $\cstar$-algebra $\A$:
\begin{enumerate}[\rm (i)]
\item $\A$ is nuclear;
\item $\A$ is amenable.
\end{enumerate}
\end{theorem}
\par 
Theorem \ref{nucam} is usually attributed to A.\ Connes---for (ii) $\Longrightarrow$ (i)---and U.\ Haagerup---for (i) $\Longrightarrow$ (ii). However, Connes proved (ii) $\Longrightarrow$ (i) only for the case where $\A^\ast$ is separable, and it took the work of several other mathematicians to establish this implication in the general case. For a self-contained proof of Theorem \ref{nucam} and references to the original literature, see \cite[Chapter 6]{LoA}. It appears that Theorem \ref{nucam} was known in its full generality by the early 1980s.
\par 
Already in 1976, S.\ Wassermann had proven (\cite{Was}):
\begin{theorem} \label{simon}
The following are equivalent for a von Neumann algebra $\M$:
\begin{enumerate}[\rm (i)]
\item $\M$ is nuclear;
\item $\M$ is \emph{subhomogeneous}, i.e., there are $n_1, \ldots, n_k \in \posints$ and commutative von Neumann algebras $\M_1, \ldots, \M_k$ such that
\[
  \M \cong M_{n_1}(\M_1) \oplus \cdots \oplus M_{n_k}(\M_k)
\]
\end{enumerate}
\end{theorem}
\par 
Of course, this means:
\begin{corollary}
Let $\Hilbert$ be a Hilbert space. Then ${\cal B}(\Hilbert)$ is amenable if and only if $\dim \Hilbert < \infty$.
\end{corollary}
\par 
In particular, ${\cal B}(\ell^2)$ is not amenable. Since $\ell^2$ is certainly the ``best behaved'' of all $\ell^p$-spaces, one is led to believe that ${\cal B}(\ell^p)$ is not amenable for any $p \in [1,\infty]$. However, there is no hope to adapt the argument we just outlined to any ${\cal B}(\ell^p)$ with $p \neq 2$: both Theorems \ref{nucam} and \ref{simon} rely heavily on the powerful tools provided by the theory of von Neumann algebras, which are not available if $p \neq 2$.

\section{The $\ell^p \oplus \ell^q$ case for $p \neq q$}

Before we start discussing the non-amenability of ${\cal B}(\ell^p)$ for various values of $p$, let's have a look at
${\cal B}(\ell^p \oplus \ell^q)$ for $p \neq q$. On might expect that dealing with ${\cal B}(\ell^p \oplus \ell^q)$ instead of ${\cal B}(\ell^p)$ makes things more complicated. The Banach algebra ${\cal B}(\ell^p \oplus \ell^q)$ is a much more complex object than ${\cal B}(\ell^p)$: for instance, the closed ideal structure of ${\cal B}(\ell^p)$ is as simple as it can get (\cite{GMF}) whereas the situation is considerably more subtle for ${\cal B}(\ell^p \oplus \ell^q)$ (\cite{SSTT}).
\par 
Still, the non-amenability of ${\cal B}(\ell^p \oplus \ell^q)$ for $p \neq q$ is surprisingly easy to obtain. 
\begin{theorem} \label{george} 
Let $p,q \in [1,\infty)$ such that $p \neq q$. Then ${\cal B}(\ell^p \oplus \ell^q)$ is not amenable.
\end{theorem}
\par 
We present the proof---which is attributed to G.\ A.\ Willis in \cite{Gro}---because it is both short and elegant and because it can be adapted to other Banach algebras with a block matrix structure (see \cite{DR}, for instance). The proof uses four ingredients, the first three of which are elementary properties of amenable Banach algebras:
\begin{itemize}
\item Every quotient of an amenable Banach algebra is again amenable.
\item A closed ideal of an amenable Banach algebra is amenable if and only if it is weakly complemented; in particular, complemented closed ideals of amenable Banach algebras are amenable.
\item Amenable Banach algebras have bounded approximate identities. 
\item \textbf{Pitt's Theorem.} \textit{Let $1 \leq q < p < \infty$. Then ${\cal B}(\ell^p,\ell^q) = {\cal K}(\ell^p,\ell^q)$ holds}  (\cite[Proposition 2.c.3]{LT}).
\end{itemize}
\begin{proof}
The Banach algebra ${\cal B}(\ell^p \oplus \ell^q)$ has a block matrix structure
\[
  {\cal B}(\ell^p \oplus \ell^q) = \begin{bmatrix} 
  {\cal B}(\ell^p) & {\cal B}(\ell^q,\ell^p) \\ {\cal B}(\ell^p,\ell^q) & {\cal B}(\ell^q) 
  \end{bmatrix}
\]
as has its closed ideal
\[
  {\cal K}(\ell^p \oplus \ell^q) = \begin{bmatrix} 
  {\cal K}(\ell^p) & {\cal K}(\ell^q,\ell^p) \\ {\cal K}(\ell^p,\ell^q) & {\cal K}(\ell^q)
  \end{bmatrix}.
\]
Without loss of generality, suppose that $p > q$. Then Pitt's Theorem yields
\[
  {\cal B}(\ell^p \oplus \ell^q) = \begin{bmatrix} 
  {\cal B}(\ell^p) & {\cal B}(\ell^q,\ell^p) \\ {\cal K}(\ell^p,\ell^q) & {\cal B}(\ell^q) 
  \end{bmatrix}.
\]
For any Banach space $E$, we write ${\cal C}(E)$ to denote its \emph{Calkin algebra}, i.e., the quotient ${\cal B}(E)/{\cal K}(E)$. In view of the foregoing, ${\cal C}(\ell^p \oplus \ell^q)$ has an upper diagonal block matrix structure, namely
\[
  {\cal C}(\ell^p \oplus \ell^q) = \begin{bmatrix} 
  {\cal C}(\ell^p) & \ast \\ 0 & {\cal C}(\ell^q)
  \end{bmatrix}.
\]
\par
Assume now that ${\cal B}(\ell^p \oplus \ell^q)$ is amenable. Then its quotient ${\cal C}(\ell^p \oplus \ell^q)$ is also amenable. Set 
\[
  I := \begin{bmatrix} 0 & \ast \\ 0 & 0 \end{bmatrix}.
\]
Then $I$ is a non-zero ideal of ${\cal C}(\ell^p \oplus \ell^q)$---because ${\cal K}(\ell^q,\ell^p) \subsetneq {\cal B}(\ell^q,\ell^p)$---, which is trivially complemented. Consequently, $I$ is an amenable Banach algebra in its own right and thus has a bounded approximate identity. Since $I^2 = \{ 0 \}$, however, this is impossible.
\end{proof}
\par 
We note that Theorem \ref{george} remains true if $\ell^p$ is replaced by $c_0$.

\section{Non-amenability of ${\cal B}(\ell^p)$ for $p=1,2,\infty$ and Ozawa's Lem\-ma}

The first to prove that ${\cal B}(\ell^p)$ was not amenable for \emph{any} $p$ other than $2$ was C.\ J.\ Read in \cite{Rea}. His proof relied on the theory of random hypergraphs. Even though the publication year of \cite{Rea} is 2006, Read had obtained his proof years earlier and, for instance, had presented it at conference in Canberra as early as January 2001. 
\par 
While Read's proof was awaiting publication, other mathematicians studied his proof, and in \cite{Pis}, G.\ Pisier presented a simplification of Read's argument---published before Read's original proof---by replacing the random hypergraphs with expanders. Eventually, N.\ Ozawa simplified Pisier's argument even further (\cite{Oza}). Expanders still play a pivotal r\^ole, but they do so in the background: one can fully understand the proof of \cite[Theorem 1.1]{Oza} without knowing what an expander is.
\par 
Ozawa's approach is remarkable because it not only gives a surprisingly simple proof for the non-amenability of ${\cal B}(\ell^1)$: it provides a unified argument for the non-amenability of several Banach algebras (\cite[Theorem 1.1]{Oza}):
\begin{theorem} \label{takthm}
The following Banach algebras are not amenable:
\begin{itemize}
\item ${\cal B}(\ell^p)$ for $p \in \{1,2,\infty\}$;
\item $\text{$\ell^\infty$-}\bigoplus_{n=1}^\infty {\cal B}(\ell^p_n)$ for $p \in [1,\infty]$;
\item the Banach algebras of all regular operators on $\ell^p$ for $p \in [1,\infty]$.
\end{itemize}
\end{theorem}
\par 
We note that, even though it is not explicitly stated in \cite{Oza}, the proof of \cite[Theorem 1.1]{Oza} also establishes the non-amenability of ${\cal B}(c_0)$ as well as of $\text{$\ell^\infty$-}\bigoplus_{n=1}^\infty {\cal B}(\ell^{p_n}_n)$ where $( p_n )_{n=1}^\infty$ is \emph{any} sequence in $[1,\infty]$.
\par 
The proof of \cite[Theorem 1.1]{Oza} is rather monolithic, but still, it has a ``front end'' and a ``back end'', which can be separated from each other, and we would like to take a closer look at the back end.
\par 
Let $\mathbb P$ denote the set of all prime numbers, and, for each $p \in \mathbb P$, write $\Lambda_p$ for the projective plane over the finite field $\ints / p \ints$. The group $\SL(3,\ints)$ has \emph{Kazhdan's property $(T)$} (\cite[Theorem 4.2.5]{BHV})---and thus, in particular, is finitely generated, by $g_1, \ldots, g_m$, say---and acts on $\Lambda_p$ through matrix multiplication. This group action, in turn, induces a unitary representation $\pi_p \!: \SL(3,\ints) \to {\cal B}(\ell^2(\Lambda_p))$. We choose a subset $S_p$ of $\Lambda_p$ with $|S_p| = \frac{|\Lambda_p|-1}{2}$ and define a unitary $\pi_p(g_{m+1}) \in {\cal B}(\ell^2(\Lambda_p))$ via
\[
 \pi_p(g_{m+1}) e_\lambda = \left\{ \begin{array}{rl} e_\lambda, & \lambda \in S_p, \\ 
                            -e_\lambda, & \lambda \notin S_p. \end{array} \right.
\]
(We would like to stress that the notation $\pi_p(g_{m+1})$ is purely symbolic and mainly for our convenience: $\pi_p(g_{m+1})$ does \emph{not} lie in $\pi_p(\SL(3,\ints))$.)
\par 
With this setup in place, we can formulate the following:
\begin{taklem}
It is impossible to find, for each $\epsilon > 0$, a number $r \in \posints$ with the following property: for each $p \in \mathbb P$, there are $\xi_{1,p}, \eta_{1,p}, \ldots, \xi_{r,p}, \eta_{r,p} \in \ell^2(\Lambda_p)$ such that $\sum_{k=1}^r \xi_{k,p} \tensor \eta_{k,p} \neq 0$ and
\begin{multline*}
  \left\| \sum_{k=1}^r \xi_{j,p} \tensor \eta_{k,p} - (\pi_p(g_j) \tensor \pi_p(g_j))(\xi_{k,p} \tensor \eta_{k,p}) 
  \right\|_{\ell^2(\Lambda_p) \Tensor \ell^2(\Lambda_p)} \\ \leq \epsilon \left\| \sum_{k=1}^r \xi_{k,p} \tensor \eta_{k,p}
  \right\|_{\ell^2(\Lambda_p) \Tensor \ell^2(\Lambda_p)} \qquad (j=1, \ldots, m+1). 
\end{multline*}
\end{taklem}
\par 
Ozawa's Lemma is not implicitly formulated in \cite{Oza}, but is proven there as part of \cite[Theorem 1.1]{Oza}. It's proof relies mainly on two ingredients:
\begin{itemize}
\item $\SL(3,\ints)$ has Kazhdan's property $(T)$, and
\item the so-called \emph{non-commutative Mazur} map is uniformly continuous (\cite[Theorem 4.1]{Oza}).
\end{itemize}
\par 
It is remarkable that the statement of Ozawa's Lemma doesn't make any explicit reference to Banach algebras at all. In \cite{Oza}, the connection with the various Banach algebras mentioned in Theorem \ref{takthm} is established through the inequalities of \cite[Lemma 2.1]{Oza}, in particular by
\[
  \sum_{n=1}^\infty \| S e_n \|_{\ell^2_N} \| T e_n^\ast \|_{\ell^2_N} \leq N \| S \| \| T \|
  \qquad (N \in \posints, \, S \in {\cal B}(\ell^p,\ell^p_N), \, T \in {\cal B}(\ell^{p'},\ell^{p'}_N))
\]
for $p =1,2,\infty$, which is no longer true for general $p \in [1,\infty]$ (\cite[Remark 2.2]{Oza}).

\section{Partial results for $p \in (1,\infty) \setminus \{ 2 \}$}

In \cite{DR}, M.\ Daws and the author studied the amenability problem for ${\cal B}(\ell^p)$ for $p \in (1,\infty) \setminus \{ 2 \}$ with the goal of proving that these algebras were not amenable. 
\par 
The starting point of their investigation was \cite[Theorem 4.1]{DR}:
\begin{theorem} \label{matti}
Let $E$ be a Banach space, and let $p \in [1,\infty)$. Then the following are equivalent:
\begin{enumerate}[\rm (i)]
\item ${\cal B}(\ell^p(E))$ is amenable;
\item $\ell^\infty({\cal B}(\ell^p(E)))$ is amenable.
\end{enumerate}
\end{theorem}
\par 
The idea of the proof of (i) $\Longrightarrow$ (ii) is to view $\ell^\infty({\cal B}(\ell^p(E)))$ as block diagonal matrices on $\ell^p(\ell^p(E)) \cong \ell^p(E)$ and then construct an approximate diagonal for these block diagonal matrices from an approximate diagonal for ${\cal B}(\ell^p(E))$.
\par 
For $E = \comps$, Theorem \ref{matti} means that ${\cal B}(\ell^p)$ is amenable if and only if $\ell^\infty({\cal B}(\ell^p))$ is amenable, which is comforting from a philosophical point of view if one believes that ${\cal B}(\ell^p)$ is not amenable: amenable Banach algebras ought to be ``small'' and $\ell^\infty({\cal B}(\ell^p))$ feels even ``larger'' than  $\text{$\ell^\infty$-}\bigoplus_{n=1}^\infty {\cal B}(\ell^p_n)$, which is known to be not amenable. But of course, that's no proof.
\par 
A different angle, however, turns out to be more promising. Since $\ell^\infty({\cal K}(\ell^p))$ is a closed ideal in $\ell^\infty({\cal B}(\ell^p))$ with a bounded approximate identity, Theorem \ref{matti} yields immediately:
\begin{corollary} \label{matticor}
Let $p \in [1,\infty)$ be such that ${\cal B}(\ell^p)$ is amenable. Then $\ell^\infty({\cal K}(\ell^p))$ is amenable.
\end{corollary}
\par 
In \cite{GJW}, Gr{\o}nb{\ae}k, Johnson, and Willis showed that ${\cal K}(E)$ is amenable whenever $E$ is an ${\cal L}^p$-space: these spaces, which were introduced in \cite{LP} and further studied in \cite{LR}, are those that ``look locally like $\ell^p$''. The precise definition is as follows:
\begin{definition}
Let $p \in [1,\infty]$ and let $\lambda \geq 1$. A Banach space $E$ is called an \emph{${\cal L}_\lambda^p$-space} if, for every finite-dimensional subspace $X$ of $E$, there is a finite-dimensional subspace $Y$ of $E$ containing $X$ with $d(Y,\ell^p_{\dim Y}) \leq \lambda$, were $d$ is the Banach--Mazur distance of Banach spaces. We call $E$ an \emph{${\cal L}^p$-space} if it is an ${\cal L}_\lambda^p$-space for some $\lambda \geq 1$.  
\end{definition}
\begin{example} \label{Lpex}
Every Banach space isomorphic to an $L^p$-space---in the usual measure theoretic sense---is an ${\cal L}^p$-space. However, if $p \in (1,\infty) \setminus \{ 2 \}$, then $\ell^p(\ell^2)$ and $\ell^2 \oplus \ell^p$ are ${\cal L}^p$-spaces, but \emph{not} isomorphic to $L^p$-spaces (\cite[Example 8.2]{LP}).
\end{example}
\par 
Since all ${\cal L}^p$-spaces ``look locally like $\ell^p$'', all ${\cal L}^p$-spaces ``look locally alike''. This has as a consequence that, whenever we have two infinite-dimensional ${\cal L}^p$-spaces $E$ and $F$, then every finite rank operator on $F$ factors through $E$ \emph{such that we can keep the norms of the factors under control}. To be precise (\cite[Lemma 4.2]{DR}):
\begin{lemma} \label{faclem}
Let $p \in [1,\infty]$, and let $E$ and $F$ be ${\cal L}^p$-spaces with $\dim E = \infty$. Then there is $C \geq 0$ such that, for each $T \in {\cal F}(F)$, there are $S \in {\cal F}(F,E)$ and $R \in {\cal F}(E,F)$ with $\| R \| \| S \| \leq C \| T \|$ and $RS = T$.
\end{lemma}
\par 
Using Lemma \ref{faclem} and \cite[Theorem 6.2]{GJW} as in the proof of \cite[Theorem 6.4]{GJW}, we obtain (\cite[Theorem 4.3]{DR}):
\begin{theorem} 
Let $p \in [1,\infty]$. Then one of the following assertions is true:
\begin{enumerate}[\rm (i)]
\item $\ell^\infty({\cal K}(E))$ is amenable for every infinite-dimensional ${\cal L}^p$-space $E$;
\item $\ell^\infty({\cal K}(E))$ is not amenable for any infinite-dimensional ${\cal L}^p$-space $E$.
\end{enumerate}
\end{theorem}
\par 
In view of Corollary \ref{matticor} and Example \ref{Lpex}, this yields:
\begin{corollary} \label{matticor2}
Let $p \in (1,\infty) \setminus \{ 2 \}$ such that ${\cal B}(\ell^p)$ is amenable. Then $\ell^\infty({\cal K}(\ell^2 \oplus \ell^p))$ is amenable.
\end{corollary}
\par 
In \cite{DR}, it was shown that $\ell^\infty({\cal K}(\ell^2 \oplus \ell^p))$ is indeed not amenable if $p=1,\infty$ (\cite[Theorem 5.3]{DR}). The proof of \cite[Theorem 5.3]{DR} involves a matrix block argument like the proof of Theorem \ref{george}. This matrix block argument, however, breaks down if $p \in (1,\infty) \setminus \{ 2 \}$ (\cite[Proposition 5.4]{DR}). Still, $\ell^\infty({\cal K}(\ell^2 \oplus \ell^p))$ is not amenable for \emph{any} $p \in [1,\infty]$: for $p = 2$, this has long been known, and for $p \in (1,\infty)$, this was shown only very recently by the author (\cite{Run}). We shall outline the argument from \cite{Run} in the two sections below.
\begin{remark}
There is a gap in the proof of \cite[Lemma 3.3]{DR}: implication (i) $\Longrightarrow$ (ii)---claimed to be routine in \cite{DR}---is, in fact, unproven. Consequently, (iii) $\Longrightarrow$ (i) of \cite[Theorem 3.2]{DR} lacks proof as well. None of this, however, affects the results from \cite{DR} discussed in this section.
\end{remark}

\section{Non-amenability of $\ell^\infty({\cal K}(\ell^2 \oplus E))$ for certain $E$}

Let $E$ be a Banach space. A \emph{basis} for $E$ is a sequence $( x_n )_{n=1}^\infty$ in $E$ such that, for every $x \in E$, there is a sequence $( \lambda_n )_{n=1}^\infty$ such that
\[
  x = \sum_{n=1}^\infty \lambda_n x_n.
\]
For each $n \in \posints$, we obtain a linear functional $x_n^\ast$ on $E$ that assigns to each $x \in E$ the coefficient $\lambda_n$ of $x_n$ in the above expansion of $x$; it is easy to see that each $x_n^\ast$ is continuous.
\par 
In this section, we indicate a proof of the following (\cite[Theorem 3.2]{Run}):
\begin{theorem} \label{ithm}
Let $E$ be a Banach space with a basis $( x_n )_{n=1}^\infty$ such that there is $C > 0$ with
\begin{equation} \tag{\mbox{$\ast$}}
  \sum_{n=1}^\infty \| S x_n \| \| T x_n^\ast \| \leq C \, N \| S \| \| T \| 
  \qquad (N \in \posints, \, S \in {\cal B}(E,\ell^2_N), \, T \in {\cal B}(E^\ast, \ell^2_N)).
\end{equation}
Then $\ell^\infty({\cal K}(\ell^2 \oplus E))$ is not amenable.
\end{theorem}
\par 
As the proof of this theorem is rather technical, we shall only sketch it; for the details see \cite{Run}.
\par 
We first record a lemma, which is easy to prove (\cite[Lemma 1.2]{Run}; compare \cite[Definiton 1.2]{Oza}):
\begin{lemma} \label{ilem}
Let $\A$ be an amenable Banach algebra, and let $e \in \A$ be an idempotent. Then, for any $\epsilon > 0$ and any finite subset $F$ of $e\A e$, there are $a_1, b_1, \ldots, a_r, b_r \in \A$ such that
\[
  \sum_{k=1}^r a_k b_k = e
\]
and 
\[
  \left\| \sum_{k=1}^r x a_k \tensor b_k - a_k \tensor b_k x \right\|_{\A \Tensor \A} < \epsilon \qquad (x \in F).
\]
\end{lemma}
\par 
We now identify $\ell^\infty({\cal K}(\ell^2 \oplus E))$ with $\ell^\infty(\mathbb{P}, {\cal K}(\ell^2 \oplus E))$ and denote it by $\A$ for simplicity. Each summand of $\A$ has a canonical block matrix structure
\[
  {\cal K}(\ell^2 \oplus E) = \begin{bmatrix} {\cal K}(\ell^2) & {\cal K}(E,\ell^2) \\
                              {\cal K}(\ell^2,E) & {\cal K}(E)
                              \end{bmatrix}.
\]
For each $p \in \mathbb P$, we embed ${\cal B}(\ell^2(\Lambda_p))$ into ${\cal K}(\ell^2 \oplus E)$ as upper left corner of the corresponding block matrix. This allows us to embed $\text{$\ell^\infty$-}\bigoplus_{p \in \mathbb P} {\cal B}(\ell^2(\Lambda_p))$ into $\A$, and in particular, we can consider
\[
  F := \left\{ ( \pi_p(g_j) )_{p \in \mathbb P} : j =1, \ldots, m+1 \right\}
\]
as a finite subset of $\A$. Furthermore, we let $\A$ act (as block diagonal matrices) on the space
\[
  \ell^2(\mathbb{P}, \ell^2 \oplus E) \cong \ell^2(\mathbb{P},\ell^2) \oplus \ell^2(\mathbb{P},E)
\]
For any $p \in \mathbb P$, let $P_p \in {\cal B}(\ell^2)$ be the canonical projection onto the first $|\Lambda_p|$ coordinates of the $p^\mathrm{th}$ $\ell^2$-summand of $\ell^2(\mathbb{P},\ell^2) \oplus \ell^2(\mathbb{P},E)$, and let $e = (P_p)_{p \in \mathbb P}$. Then $e$ is an idempotent in $\A$ with
\[
  e \A e =   \text{$\ell^\infty$-}\bigoplus_{p \in \mathbb P} {\cal B}(\ell^2(\Lambda_p)).
\]
\par 
The rest of the proof resembles that of \cite[Theorem 1.1]{Oza}, with Lemma \ref{ilem} replacing \cite[Definition 1.2]{Oza} and ($\ast$) providing the required estimates similar to \cite[Lemma 2.1]{Oza}: if $\A$ is amenable, then Lemma \ref{ilem} is true, which eventually contradicts Ozawa's Lemma.

\section{Non-amenability of ${\cal B}(\ell^p)$ for $p \in (1,\infty)$} 

With Theorem \ref{ithm} proven, we shall now see that its hypotheses are indeed satisfied if $E = \ell^p$ with $p \in (1,\infty)$.
\par 
Let $N \in \posints$, let $S \in {\cal B}(\ell^p,\ell^2_N)$, and let $T \in {\cal B}(\ell^{p'}, \ell^2_N)$. We can identify
\[
  {\cal B}(\ell^p,\ell^2_N) = \ell^{p'} \wTensor \ell^2_N = \ell^{p'} \tensor \ell_N^2
  \qquad\text{and}\qquad 
  {\cal B}(\ell^{p'},\ell^2_N) = \ell^p \wTensor \ell^2_N = \ell^p \tensor \ell_N^2,
\]
where $\wTensor$ is the injective tensor product. On the other hand, we can also---algebraically---identify
\[
  \ell^{p'}(\ell^2_N) = \ell^{p'} \tensor \ell_N^2 \qquad\text{and}\qquad \ell^p(\ell^2_N) = \ell^p \tensor \ell_N^2
\]
and thus view $S$ and $T$ as elements of $\ell^{p'}(\ell^2_N)$ and $\ell^p(\ell^2_N)$, respectively. H\"older's inequality then yields immediately that 
\begin{equation} \tag{\mbox{$\ast\ast$}}
  \sum_{n=1}^\infty \| S e_n \| \| T e_n^\ast \| \leq \| S \|_{\ell^{p'}(\ell^2_N)} \| T \|_{\ell^p(\ell^2_N)} 
\end{equation}
where $( e_n )_{n=1}^\infty$ is the canonical basis of $\ell^p$.
\par 
In order to prove ($\ast$) for $E = \ell^p$, we thus need to relate the norms on $\ell^{p'} \wTensor \ell^2_N$ and $\ell^{p'}(\ell^2_N)$ and on $\ell^p \wTensor \ell^2_N$ and $\ell^p(\ell^2_N)$, respectively. This relation is provided by the notion of a $p$-summing operator (see \cite{DJT}, for instance). 
\par 
For convenience, we recall:
\begin{definition} \label{psum}
Let $p \in [1,\infty)$, and $E$ and $F$ be Banach spaces. A linear map $T \!: E \to F$ is called \emph{$p$-summing} if the amplification $\id_{\ell^p} \tensor T \!: \ell^p \tensor E \to \ell^p \tensor F$ extends to a bounded map from $\ell^p \wTensor E$ to $\ell^p(F)$. The operator norm of $\id_{\ell^p \tensor T} \!: \ell^p \check{\otimes} E \to \ell^p(F)$ is called the \emph{$p$-summing norm} of $T$ and denoted by $\pi_p(T)$.
\end{definition}
\par 
Every finite rank operator is $p$-summing (\cite[Proposition 2.3]{DJT}), but the identity on a Banach space $E$ is $p$-summing if and only if $\dim E < \infty$ (\cite[Weak Dvoretzky--Rogers Theorem 2.18]{DJT}).
\par 
The $p$-summing norm of $\id_{\ell^2_N}$ was computed by Y.\ Gordon in \cite{Gor} for any $p \in [1,\infty)$. For our purpose, his exact formula is not important, but rather a consequence of it (\cite[Theorem 5]{Gor}), namely the asymptotic estimate
\[
  \pi_p(\id_{\ell^2_N}) \sim \sqrt{N}.
\]
This means, in particular, that there is $C_p > 0$ such that
\[
  \pi_p(\id_{\ell^2_N}) \leq C_p \, \sqrt{N}.
\]
If $p \in (1,\infty)$, we apply the same asymptotic estimate to $\pi_{p'}(\id_{\ell^2_N})$ and obtain $C_{p'} > 0$ such that
\[
  \pi_{p'}(\id_{\ell^2_N}) \leq C_{p'} \, \sqrt{N}.
\]
Together with ($\ast\ast$) and Definition \ref{psum}, these estimates yield:
\begin{lemma}
 Let $p \in (1,\infty)$. Then there is $C > 0$ such that
\[
  \sum_{n=1}^\infty \| S e_n \|_{\ell^2_N} \| T e^\ast_n \|_{\ell^2_N} \leq C \, N \| S \| \| T \| 
  \qquad (N \in \posints, \, S \in {\cal B}(\ell^p, \ell^2_N), \, T \in {\cal B}(\ell^{p'}, \ell^2_N)).
\]
\end{lemma}
\par 
Hence, $\ell^p$ with $p \in (1,\infty)$ satisfies the hypothesis of Theorem \ref{ithm}, so that $\ell^\infty({\cal K}(\ell^2 \oplus \ell^p))$ is not amenable. By Corollary \ref{matticor}, this means that ${\cl B}(\ell^p)$ is not amenable either.
\par 
With a little more effort, we obtain:
\begin{theorem} \label{nonamthm}
Let $p \in (1,\infty)$, and let $E$ be an ${\cal L}^p$-space. Then ${\cal B}(\ell^p(E))$ is not amenable.
\end{theorem}
\par 
Of course, $E \cong \ell^p(E)$ for $E = \ell^p$ or $E = L^p[0,1]$, so that ${\cal B}(E)$ cannot be amenable for those $E$ as a consequence of Theorem \ref{nonamthm}.
\par 
Let $E$ be an infinite-dimensional, separable $L^p$-space. Then measure theoretic arguments yield that $E$ must be isometrically isomorphic to one of the following to $\ell^p$, $L^p[0,1]$, $L^p[0,1] \oplus \ell^p$, and $L^p[0,1] \oplus \ell^p_N$ for with $N \in \posints$ (\cite[Proposition III.A.1]{Woj}). In the first three cases, we have $\ell^p(E) \cong E$, so that ${\cal B}(E)$ is not amenable. For the third case, note that the canonical map ${\cal B}(\ell^p_N,L^p[0,1]) \Tensor {\cal B}(L^p[0,1],\ell^p_N)$ into ${\cal B}(\ell^p_N)$ is onto. Making using of the non-amenability of ${\cal B}(L^p[0,1])$ and \cite[Theorem 6.2]{GJW}, we conclude that ${\cal B}(E)$ is non-amenable in this case, too.
\par 
Hence, we obtain:
\begin{corollary}
Let $E$ be an infinite-dimensional, separable $L^p$-space. Then ${\cal B}(E)$ is not amenable.
\end{corollary}
\par 
Alternatively, we could have used Banach space theoretic methods to prove that up to---not necessarily isometric---isomorphism $\ell^p$ and $L^p[0,1]$ are the only in\-fi\-nite-di\-men\-sio\-nal, separable $L^p$-spaces (see \cite[p.\ 83]{Woj}).

\section{Questions}

As a consequence of the solution to the ``scalar plus compact'' problem, there is an infinite-dimensional Banach space $E$ such that ${\cal B}(E)$ is amenable, while ${\cal B}(E)$ is not amenable for $E = \ell^p$ with $p \in [1,\infty]$ or for $E = c_0$. If $E$ is such that $E^\ast$ has the approximation property, than ${\cal K}(E)$ is an ideal in ${\cal B}(E)$ with a bounded approximate identity (\cite[Corollary 3.1.5]{LoA}) and thus, provided that ${\cal B}(E)$ is amenable, must be amenable, too. The negative results from \cite{GJW} and \cite{BG} about the (non-)amenability of ${\cal K}(E)$ will therefore, by all likelihood, lead to further examples of non-amenable ${\cal B}(E)$.
\par 
The obvious big open problem has now become:
\begin{question}
Which are the Banach spaces $E$ for which ${\cal B}(E)$ is amenable? 
\end{question}
\par 
There is very little hope that this problem will ever be answered in its entirety, but maybe restricting to certain classes of well behaved Banach spaces will lead to results.
\par 
The space constructed in $E$ is a predual of $\ell^1$ and thus, in particular, is not reflexive. This prompts:
\begin{question}
Is there a reflexive, infinite-dimensional Banach space $E$ such that ${\cal B}(E)$ is amenable?  
\end{question}
\par 
If ${\cal B}(E)$ is amenable, then so is the Calkin algebra ${\cal C}(E)$. If $E$ is an $\ell^p$-space, then ${\cal C}(E)$ is not separable and not amenable, i.e., it is ``huge''. If $E$ is the space from \cite{AH}, however, then ${\cal C}(E)$ is ``tiny'', namely one-dimensional. The one-dimensional case is not the only situation in which an amenable Calkin algebra can occur: let $E$ be the space from \cite{AH}, and let $n \in \posints$. Then ${\cal B}(E^n) \cong M_n({\cal B}(E))$ is amenable with ${\cal C}(E^n) \cong M_n$.
But still, we don't have examples where ${\cal C}(E)$ is amenable and infinite-dimensional.
\par 
So, we ask:
\begin{question} \label{calkin}
Is there a Banach space $E$ such that ${\cal C}(E)$ is amenable and in\-fi\-nite-di\-men\-sio\-nal?
\end{question}
\par 
In \cite{DLW}, a Banach space $E$ is constructed with the property that there is a closed ideal ${\cal J}(E)$ of ${\cal B}(E)$ with ${\cal B}(E) / {\cal J}(E) \cong \ell^\infty$. Could this construction possibly be modified in such a way that ${\cal B}(E) / {\cal K}(E) \cong \ell^\infty$? If so, it would provide a positive answer to Question \ref{calkin}.
\renewcommand{\baselinestretch}{1.0}

\renewcommand{\baselinestretch}{1.2}
\dated
\vfill
\begin{tabbing}
\textit{Author's address}: \= Department of Mathematical and Statistical Sciences \\
\> University of Alberta \\
\> Edmonton, Alberta \\
\> Canada T6G 2G1 \\[\medskipamount]
\textit{E-mail}: \> \texttt{vrunde@ualberta.ca} \\[\medskipamount]
\textit{URL}: \> \texttt{http://www.math.ualberta.ca/$^\sim$runde/}
\end{tabbing}           

\end{document}